\documentclass[11pt,a4paper,twoside]{article}
\usepackage{amssymb}
\usepackage{amsmath}
\usepackage{amsthm}
\usepackage{mathtools}
\usepackage{enumerate}
\usepackage[utf8]{inputenc}
\usepackage[T1]{fontenc}
\usepackage[all]{xy}
\usepackage{datetime}
\usepackage{mathrsfs}
\usepackage{graphicx}
\usepackage{memhfixc}
\usepackage{aliascnt}
\usepackage[linktocpage=true,bookmarksnumbered,bookmarksopen,bookmarksopenlevel=0,colorlinks]{hyperref}
\usepackage{color}
\definecolor{MyCiteColor}{rgb}{0,0.45,0}
\definecolor{MyLinkColor}{rgb}{0.0,0.08,0.45}
\definecolor{MyAnchorColor}{rgb}{0,0.08,0.45}    
\hypersetup{
 pdftitle={Bornologically isomorphic representations of tensor distributions},
 pdfauthor={Eduard Nigsch},
 pdfsubject={},
 citecolor=MyCiteColor,
 linkcolor=MyLinkColor,
 anchorcolor=MyAnchorColor
}

\newtheoremstyle{indentheorem}%
	{3pt}
	{3pt}
	{}
	{}
	{\itshape}
	{:}
	{.5em}
	{}
\theoremstyle{indentheorem}

\theoremstyle{plain}
\newtheorem{theorem}{Theorem}
\newtheorem{corollary}[theorem]{Corollary}
\newtheorem{proposition}[theorem]{Proposition}
\newtheorem{lemma}[theorem]{Lemma}

\theoremstyle{definition}
\newtheorem{definition}[theorem]{Definition}
\theoremstyle{remark}
\newtheorem{remark}[theorem]{Remark}

\newcommand{\pd}{\partial}

\newcommand{\bC}{\mathbb{C}}
\newcommand{\bE}{\mathbb{E}}
\newcommand{\bF}{\mathbb{F}}

\newcommand{\bK}{\mathbb{K}}
\newcommand{\bN}{\mathbb{N}}
\newcommand{\bR}{\mathbb{R}}

\newcommand{\bZ}{\mathbb{Z}}

\newcommand{\ccD}{\mathcal{D}}

\newcommand{\cP}{\mathcal{P}}

\newcommand{\cS}{\mathcal{S}}

\newcommand{\sC}{\mathscr{C}}

\newcommand{\sT}{\mathscr{T}}

\newcommand{\csub}{\subset \subset}
\newcommand{\coleq}{\mathrel{\mathop:}=}

\DeclareMathOperator{\id}{id}
\DeclareMathOperator{\Vol}{Vol}

\DeclareMathOperator{\pr}{pr}

\newcommand{\cTrsM}{\mathcal{T}^r_s(M)}

\newcommand{\cTsrM}{{\ensuremath{\mathcal{T}^s_r(M)}}}
\newcommand{\DprsM}{\ccD'^r_s(M)}

\newcommand{\ocM}{{\Omega^n_c(M)}}

\newcommand{\DpM}{\ccD'(M)}

\newcommand{\TsrM}{\mathrm{T}^s_r(M)}

\newcommand{\Lin}{\mathrm{L}}

\providecommand{\norm}[1]{\left\lVert#1\right\rVert}
\providecommand{\abso}[1]{\left\lvert#1\right\rvert}

\newcommand{\fp}{\mathfrak{p}}
\newcommand{\Linb}{\mathrm{L}^b}
\newcommand{\Linc}{\mathrm{L}^c}

\newcommand{\CinfM}{{C^\infty(M)}}
\newcommand{\sA}{\mathscr{A}}
\newcommand{\sB}{\mathscr{B}}
\newcommand{\sD}{\mathscr{D}}

\title{Bornologically isomorphic representations of distributions on manifolds}
\author{E. Nigsch}

\begin{document}
 
\maketitle

\begin{abstract}
Distributional tensor fields can be regarded as multilinear mappings with distributional values or as (classical) tensor fields with distributional coefficients. We show that the corresponding isomorphisms hold also in the bornological setting.
\end{abstract}

\section{Introduction}

Let $\DpM \coleq \Gamma_c(M, \Vol(M))'$ and $\DprsM \coleq \Gamma_c(M, \TsrM \otimes \Vol(M))'$ be the strong duals of the space of compactly supported sections of the volume bundle $\Vol(M)$ and of its tensor product with the tensor bundle $\TsrM$ over a manifold; these are the spaces of scalar and tensor distributions on $M$ as defined in \cite{GKOS,Hoermander}. A property of the space of tensor distributions which is fundamental in distributional geometry is given by the $\CinfM$-module isomorphisms
\begin{align}\label{the_isos}
\DprsM &\cong \Lin_{\CinfM}(\cTsrM, \DpM) \cong \cTrsM \otimes_\CinfM \DpM
\end{align}
(cf.~\cite[Theorem 3.1.12 and Corollary 3.1.15]{GKOS}) where $\CinfM$ is the space of smooth functions on $M$. In \cite{global2} a space of Colombeau-type nonlinear generalized tensor fields was constructed. This involved handling smooth functions (in the sense of convenient calculus as developed in \cite{KM}) in particular on the $\CinfM$-module tensor products $\cTrsM \otimes_\CinfM \DpM$ and $\Gamma(E) \otimes_\CinfM \Gamma(F)$, where $\Gamma(E)$ denotes the space of smooth sections of a vector bundle $E$ over $M$. In \cite{global2}, however, only minor attention was paid to questions of topology on these tensor products. One can circumvent this issue by declaring the occurring algebraic isomorphisms to be homeomorphisms, but this is not truly satisfying.

The aim of this article is to show that the isomorphisms \eqref{the_isos} are even bornological (not topological) isomorphisms. Naturally this involves the right choice of topologies on spaces of $\CinfM$-linear mappings and on tensor products of locally convex modules. Because there is only fragmentary literature available on tensor products of locally convex modules 
we will be rather explicit in our treatment.

A bornological isomorphism is enough for the applications in \cite{global2} because the notion of smoothness employed there depends only on the bornology of the respective spaces.

After some preliminaries in Section \ref{sec_prel} we will review inductive locally convex topologies and final convex bornologies defined by arbitrary (i.e., non-linear) mappings in Section \ref{sec_indtop}. Then the bornological and projective tensor product of locally convex and bounded modules are defined and usual properties of tensor products established in this setting (Section \ref{sec_tensprod}). Afterwards we will describe the natural Fr\'echet topology on spaces of sections (Section \ref{sec_sectop}) and show that some canonical algebraic isomorphisms for spaces of sections are homeomorphisms as well. As a main result we obtain that the classical isomorphism of section spaces $\Gamma(E) \otimes_\CinfM \Gamma(F) \cong \Gamma(E \otimes F)$ is a homeomorphism if one uses the projective tensor product, while for compactly supported sections one has to use the bornological tensor product (Section \ref{sec_iso}).

In Section \ref{sec_distrib} we obtain the desired result \eqref{the_isos} on distributions in the bornological setting. It does not work in the topological setting for two reasons: first, multiplication of distributions by smooth functions is jointly bounded (Lemma \ref{dpmult}) but only separately continuous; and second, the bornological tensor product has better algebraic properties (Remark \ref{algprop}).

\section{Preliminaries}\label{sec_prel}

Our basic references are \cite{Jarchow, Schaefer, Treves} for topological vector spaces, \cite{henri} for bornological spaces, \cite{Lang} for differential geometry, and \cite{Bourbaki} for algebra.

All locally convex spaces are over the field $\bK$ which is either $\bR$ or $\bC$ and will be assumed to be Hausdorff. In the non-Hausdorff case we speak of a topological vector space with locally convex topology.

We will use the following notation: for vector spaces $E_1, \dotsc, E_n, F$, $\Lin(E_1, \dotsc, E_n; F)$ is the vector space of all $n$-multilinear mappings from $E_1 \times \dotsc \times E_n$ to $F$. We write $\Lin(E,F)$ instead of $\Lin(E; F)$. $F^* = \Lin(F, \bK)$ denotes the algebraic dual of $F$. If we have locally convex spaces then $\Linb(E_1,\dotsc,E_n; F)$ denotes the space of bounded multilinear mappings equipped with the topology of uniform convergence on bounded sets as in \cite[Section 5]{KM}. $\Linc(E_1, \dotsc, E_n; F)$ is the subspace of all continuous mappings equipped with the subspace topology. $E' = \Linc(E, \bK)$ denotes the topological dual of a locally convex space $E$ equipped with the strong dual topology.

For any $R$-modules $M_1, \dotsc, M_n$ and $N$, $\Lin_R(M_1, \dotsc, M_n; N)$ is the space of $R$-multilinear mappings from $M_1 \times \dotsc \times M_n$ to $N$. If these are  locally convex modules as in Definition \ref{lcmod} below (with $\bK \subseteq R$) the subspace $\Linb_R(M_1, \dotsc, M_n; N) \subseteq \Linb(M_1, \dotsc, M_n; N)$ is the space of bounded $R$-multilinear mappings from $M_1 \times \dotsc \times M_n$ to $N$ equipped with the subspace topology. We also equip the subspace $\Linc_R(E_1, \dotsc, E_n; F) \subseteq \Linb_R(E_1, \dotsc, E_n; F)$ of all continuous such mappings with the subspace topology.

Let $M$ be a right module and $N$ a left module over a ring $A$ and $E$ a $\bZ$-module. A $\bZ$-bilinear mapping $f: M \times N \to E$ is called $A$-\emph{balanced} if $f(m a, n) = f(m, a n)$ for all $a \in A$, $m \in M$ and $n \in N$. Because we want to obtain vector spaces we assume that $\bK \subseteq A$ and $E$ is a vector space and denote by $\Lin^A(M,N; E)$ the subspace of $\Lin(M,N;E)$ consisting of $A$-balanced $\bK$-bilinear mappings. If $M,N$ are locally convex or bounded modules and $E$ a locally convex space $\Lin^{A,b}(M,N;E)$ and $\Lin^{A,c}(M,N;E)$ are the subspaces of bounded and continuous mappings, respectively.

Note that in all the above cases the subspace topology is again the topology of uniform convergence on bounded sets, which is locally convex and Hausdorff.

All manifolds are supposed to be finite dimensional, second countable, and Hausdorff. Vector bundles are always finite dimensional. The space of sections of a vector bundle $E$ over a manifold $M$ is denoted by $\Gamma(M, E)$, the space of compactly supported sections by $\Gamma_c(M, E)$, and the space of sections with compact support in a set $K \subseteq M$ by $\Gamma_{c,K}(M, E)$.

\section{Final and initial structures in topology and bornology}\label{sec_indtop}

The projective tensor product $E \otimes F$ of two locally convex spaces carries the inductive locally convex topology with respect to the canonical bilinear mapping $\otimes\colon E \times F \to E \otimes F$. Now there are several shortcomings in standard references, of which we mention two: first, in \cite[Section 15.1]{Jarchow} $E \otimes F$ is said to be endowed with the \emph{finest topology} (not locally convex topology) which makes $\otimes$ continuous, and it is claimed that this topology is locally convex by referring to the corresponding proposition about the \emph{projective} topology, which does not apply here; furthermore, the universal property of the inductive locally convex topology is only mentioned for linear mappings (\cite[Section 6.6]{Jarchow}) but $\otimes$ is bilinear. Second, \cite{Treves} correctly takes the finest locally convex topology on $E \otimes F$ such that $\otimes$ is continuous but does not show its universal property.

The construction of the projective tensor product in \cite{Schaefer} is done directly without reference to the inductive topology, which works for the purpose. We will give the inductive locally convex topology as well as its universal property also with respect to nonlinear mappings. Additionally we will have to consider the bornological tensor product which we will introduce from the topological and the bornological point of view. Similarly to the topological case, in the standard reference \cite{henri} on bornologies the final vector  or convex bornology, respectively, is only treated with respect to \emph{linear} mappings; we will show that the construction of the final vector or convex bornology, respectively, outlined there works for arbitrary mappings, too.

Given a set $E$, topological spaces $E_j$ and mappings $T_j: E \to E_j$ we denote the projective topology on $E$ defined by these mappings by $\sT_i$. In the linear or locally convex case there is no generalization to arbitrary mappings: given any Hausdorff topological vector space $E$ the projective topology with respect to the constant mappings $f_x(y) \coleq 0$ if $y=x$ and $x$ if $y \ne x$ for all $x \in E$ is the discrete topology which cannot be linear. As any linear topology making all $f_x$ continuous would be finer than the projective topology there can be none. However, in the case of the inductive topology we can allow arbitrary mappings:

\begin{lemma}\label{l2}Let $(E_j)_j$ be a family of topological vector spaces, $E$ a vector space and $S_j\colon E_j \to E$ an arbitrary mapping for each $j$. Then there is a finest linear topology $\sT_l$ on $E$ such that all $S_j$ are continuous into $[E, \sT_l]$. A linear mapping $T$ from $[E, \sT_l]$ into any topological vector space is continuous if and only if all $T \circ S_j$ are so. The same statements hold for all topologies being locally convex.
\end{lemma}

\begin{proof}
$\sT_l$ is obtained as the projective topology defined by the family of identity mappings from $E$ into all linear (locally convex) topologies $\sT$ on $E$ such that all $S_j$ are continuous w.r.t.\ $\sT$; this family contains at least the trivial topology given by $\sT = \{\emptyset, E\}$.

Given $[F, \sT]$ with $\sT$ a linear (locally convex) topology and $T \in \Lin(E,F)$, $T\colon [E, T^{-1}(\sT)] \to [F, \sT]$ is continuous; all $S_j$ are continuous into the linear topology $T^{-1}(\sT)$ 
(the preimage of a linear or locally convex topology is of the same type) because $S_j^{-1}(T^{-1}(\sT)) = (T \circ S_j)^{-1}(\sT)$ is a family of open sets by assumption, thus $\sT_l$ is finer than $T^{-1}(\sT)$ and $T\colon [E, \sT_l] \to [E, T^{-1}(\sT)] \to [F, \sT]$ is continuous.
\end{proof}

$\sT_l$ is called the inductive linear (locally convex) topology defined by the family $(S_j)_j$. We now will consider the bornological setting. The following is easily seen from the respective definitions in \cite{henri}.

\begin{lemma}\label{bornbasis}Let $X$ be a set and $\sB_0$ a family of subset of $X$. Then $\sB_0$ is a base for a bornology on $X$ if and only if $\sB_0$ covers $X$ and every finite union of elements of $\sB_0$ is contained in a member of $\sB_0$. If $X$ is a vector space, $\sB_0$ is a base for a vector bornology on $X$ if and only if additionally it every finite sum of elements of $\sB_0$ is contained in a member of $\sB_0$,
every homothetic image (scalar multiple) of an element of $\sB_0$ is contained in a member of $\sB_0$, and every circled hull of an element of $\sB_0$ is contained in a member of $\sB_0$. $\sB_0$ is a base for a convex bornology on $X$ if and only if it is a base for a vector bornology and every convex hull of elements of $\sB_0$ is contained in a member of $\sB_0$.
\end{lemma}

\begin{lemma}\label{borngen}Let $X$ be a set and $\sA$ be any family of subsets of $X$. Define the family $\sD \coleq \sA \cup \{\, \{x\}\ |\ x \in X\,\}$. Then a base of the bornology generated by $\sA$ is given by all finite unions of elements of $\sD$.
If $X$ is a vector space a base of the vector bornology generated by $\sA$ is given by all subsets of $X$ which can be obtained from elements of $\sD$ by any finite combination of finite sums, finite unions, homothetic images, and circled hulls.
For the convex bornology generated by $\sA$ one further has to include convex hulls.
\end{lemma}
\begin{proof}
%
%

Let $\sB_0$ be the family of all subsets of $X$ which can be obtained from elements of $\sD$ by the respective operations. By Lemma \ref{bornbasis} $\sB_0$ is a base for a  bornology (vector bornology, convex bornology) on $X$. Any bornology (vector bornology, convex bornology) $\sC$ on $X$ containing $\sA$ and thus $\sD$ is closed under the same operations which are applied to elements of $\sD$ in order to construct $\sB_0$, $\sB_0$ is finer than $\sC$. This means that $\sB_0$ is a base of the bornology (vector bornology, convex bornology) generated by $\sA$.
\end{proof}

\begin{proposition}Let $X$ be a set and $[X_i, \sB_i]$ bornological sets with any mappings $v_i\colon X_i \to X$. Let $\sB_f$ be the bornology on $X$ generated by the family $\sA = \bigcup_{i \in I} v_i(\sB_i)$. Then $\sB_f$ is the finest bornology on $X$ such that all $v_i$ are bounded. A mapping $v$ from $[X, \sB_f]$ into a bornological set $[Y, \sC]$ is bounded if and only if all compositions $v \circ v_i$ are bounded.

The same holds analogously for the vector (convex) bornology on a vector space $X$ generated by $\sA$ and a linear mapping $v$ into a vector (convex) bornological space $[Y, \sC]$ where the $v_i$ can be arbitrary.
 \end{proposition}
\begin{proof}
Any bornology $\sC$ on $X$ such that the $v_i$ are bounded has to contain $\bigcup_i v_i(\sB_i)$. By definition $\sB_f$ is the finest bornology (vector bornology, convex bornology) containing this set so $\sB_f$ is the finest bornology of its type such that all $v_i$ are bounded.

If $v$ is bounded the $v \circ v_i$ trivially are so. Conversely, assume that all the $v \circ v_i$ are bounded into $[Y, \sC]$. Let $\sC_f$ be the bornology (vector bornology, convex bornology) on $Y$ generated by $\bigcup_i (v \circ v_i)(\sB_i)$. Because $\sC_f$ is finer than $\sC$ it suffices to show that $v$ is bounded into $\sC_f$. Because $v$ is linear it maps the base of $\sB_f$ given by Lemma \ref{borngen} to a base of $\sC_f$ which implies that $v$ is bounded into $\sC_f$.
\end{proof}

We call $\sB_f$ the final bornology (vector bornology, convex bornology) defined by the $v_i$. Given any locally convex topology $\sT$ we denote by $\prescript{\mathrm{b}}{}\sT$ its von Neumann bornology (\cite[1:3]{henri}). Conversely, $\prescript{\mathrm{t}}{}\sB$ denotes the locally convex topology associated with a convex bornology $\sB$ (\cite[4:1]{henri}). Whenever we talk of boundedness of a mapping from or into a topological vector space with locally convex topology it is meant with respect to its von Neumann bornology.

In order to relate the bornological to the topological setting we will make use of the following Lemma.

\begin{lemma}\label{borninftop}Let $E_i$ be a topological vector space with locally convex topology and $v_i\colon E_i \to F$ an arbitrary mapping into a vector space $F$ for each $i$. Denote by $\sT_f$ the finest locally convex topology on $F$ such that each $v_i$ is bounded and by $\sB_f$ the finest convex bornology on $F$ such that each $v_i$ bounded. Then $\sT_f = \prescript{\mathrm{t}}{}\sB_f$. 
\end{lemma}
\begin{proof}
We show that each $v_i$ is bounded into $\prescript{\mathrm{t}}{}\sB_f$, which implies that $\sT_f$ is finer than $\prescript{\mathrm{t}}{}\sB_f$. Given a bounded set $B$ in $E_i$ its image $v_i(B)$ is bounded in $\sB_f$ by assumption. Because the identity $[F, \sB_f] \to [F, \prescript{\mathrm{t}}{}\sB_f]$ is bounded (\cite[4:1]{henri}), $v_i(B)$ is bounded in $\prescript{\mathrm{t}}{}\sB_f$.

Conversely, the identity $[F, \sB_f] \to [F, \sT_f]$ is bounded if and only if all mappings $v_i\colon E_i \to [F, \sT_f]$ are bounded, which is the case by construction, thus $\sB_f$ is finer than $\prescript{b}{}\sT_f$. By definition of the locally convex topology associated with a convex bornology (\cite[4:1'2]{henri}) $\prescript{\mathrm{t}}{}\sB_f$ is finer than $\sT_f$.
\end{proof}
By \cite[4:1'5 Definition (2) and Lemma (2)]{henri} we obtain

\begin{corollary}\label{borninftop_born}In the situation of Lemma \ref{borninftop} $\sT_f$ is bornological.
\end{corollary}

We recall that a bornological vector space is separated (i.e., $\{0\}$ is the only bounded vector subspace) if and only if $\{0\}$ is Mackey-closed (\cite[2:11 Proposition (1)]{henri}). By \cite[1:4'2 Proposition (1)]{henri} and \cite[Chapter 36]{Treves}) one immediately obtains the following (the converse does not hold, in general):

\begin{lemma}\label{hausdorffsepp}Let $[E,\sT]$ be a topological vector space with locally convex topology. If $\sT$ is Hausdorff then $\prescript{\mathrm{b}}{}\sT$ is separated.
\end{lemma}


\section{Tensor product of locally convex modules}\label{sec_tensprod}

\subsection{\texorpdfstring{Bornological and projective tensor product\\ of locally convex spaces}{Bornological and projective tensor product of locally convex spaces}}

We cite the following definitions of the tensor product of locally convex spaces (\cite[5.7]{KM}, \cite[Definition 43.2]{Treves}).
\begin{definition}Let $E,F$ be locally convex spaces.
\begin{enumerate}[(i)]
 \item The \emph{bornological tensor product} of $E$ and $F$ is the algebraic tensor product $E \otimes F$ of vector spaces equipped with the finest locally convex topology such that the canonical mapping $(x,y) \to x \otimes y$ from $E \times F$ into $E \otimes F$ is bounded. $E \otimes F$ with this topology is denoted by $E \otimes_\beta F$.
 \item The \emph{projective tensor product} of $E$ and $F$ is the algebraic tensor product $E \otimes F$ of vector spaces equipped with the finest locally convex topology such that the canonical mapping $(x,y) \to x \otimes y$ from $E \times F$ into $E \otimes F$ is continuous. $E \otimes F$ with this topology is denoted by $E \otimes_\pi F$.
\end{enumerate}
\end{definition}
Both $E \otimes_\beta F$ and $E \otimes_\pi F$ are Hausdorff (\cite[Section 15.1 Proposition 3]{Jarchow}). By Corollary \ref{borninftop_born} $E \otimes_\beta F$ is bornological. For any locally convex space $G$ there are bornological isomorphisms of locally convex spaces
\begin{equation}\label{bornisolb}
\Linb(E \otimes_\beta F, G) \cong \Linb(E,F;G ) \cong \Linb(E, \Linb(F,G))
\end{equation}
where the first isomorphism is given by the transpose of the canonical bilinear mapping $\otimes\colon E \times F \to E \otimes_\beta F$ and the second one by the exponential law \cite[5.7]{KM}. Consequently, a bilinear mapping $E \times F \to G$ is bounded if and only if the associated linear mapping $E \otimes_\beta F \to G$ is bounded. For the projective tensor product, however, the algebraic isomorphism of vector spaces (\cite[Proposition 43.4]{Treves})
\begin{equation}\label{bornisolc}
\Linc(E \otimes_\pi F, G) \cong \Linc(E,F;G)
\end{equation}
is not continuous and $\Linc(E,F;G)$ is not isomorphic to $\Linc(E,\Linc(F,G))$, in general,
but we have the universal property that a bilinear mapping $E \times F \to G$ is continuous if and only if the associated linear mapping $E \otimes_\pi F \to G$ is continuous.

\subsection{Vector space structures on rings and modules}

We will now define the notion of bounded resp.~locally convex algebra and module (cf.~\cite{capbracket}, \cite[Chapter 6]{Navarro}, and \cite[Chapter II]{capelle}).

Our notion of locally convex $A$-module will require $A$ to be a $\bK$-algebra. Let $R$ be a nonzero ring and $\iota\colon \bK \to R$ any mapping. Define the action of $\bK$ on $R$ (scalar multiplication) by the mapping $\bK \times R \to R$, $(\lambda, r) \mapsto \iota(\lambda) \cdot r$. This turns $R$ into a vector space over $\bK$ if and only if $\iota$ is a ring homomorphism. By \cite[I \S 9.1 Theorem 2]{Bourbaki} the subring $\iota(\bK)$ of $R$ then is a field and $\iota$ is an isomorphism of $\bK$ onto $\iota(\bK)$. Because $\bK$ is commutative $R$ is an associative unital algebra over $\bK$.

\begin{definition} We call a locally convex space $A$ over $\bK$ with a bilinear multiplication mapping $A \times A \to A$ a \emph{bounded algebra} or a \emph{locally convex algebra} over $\bK$, respectively, if this multiplication is bounded or continuous, respectively.
\end{definition}


\begin{definition}\label{lcmod}Let $A$ be a bounded (locally convex) algebra over $\bK$. A left $A$-module $M$ carrying a topology which is locally convex with respect to the vector space structure on $M$ induced by the subring $\bK \subseteq A$ is called a \emph{bounded (locally convex) left module} if module multiplication $A \times M \to M$ is bounded (continuous).
\end{definition}

The definition for right modules is analogous.

\begin{remark}One can also define a bounded (locally convex) left module $M$ over $A$ as a topological vector space $M$ with locally convex topology together with a $\bZ$-bilinear bounded (continuous) mapping $A \times M \to M$, $(a,m) \mapsto a \cdot m$ such that $a \cdot (b \cdot m) = (a b) \cdot m$ and $1 \cdot m = m$.
\end{remark}

\subsection{\texorpdfstring{Bornological and projective tensor product\\ of locally convex and bounded modules}{Bornological and projective tensor product of locally convex and bounded modules}}

We will from now on assume that the algebra $A$ contains $\bK$ as a subring of its center. This is necessary for the tensor product $M \otimes_A N$ of $A$-modules and the quotient $M \otimes_\bK N/J_0$ with $J_0$ as defined below to be vector spaces.

Let $A$ be a bounded algebra over $\bK$, $M$ a right bounded $A$-module and $N$ a left bounded $A$-module. Define $J_0$ as the sub-$\bZ$-module of $M \otimes_\bK N$ generated by all elements of the form $ma \otimes n - m \otimes an$ with $a \in A$, $m \in M$ and $n \in N$. The $\bK$-vector spaces $M \otimes_A N$ and $(M \otimes_\bK N) / J_0$ are isomorphic \cite[Theorem I.5.1]{capelle}, but in order to obtain a Hausdorff space we need to take the quotient with respect to the closure $J$ of $J_0$ in $M \otimes_\beta N$, which again is a sub-$\bZ$-module of $M \otimes_\beta N$. We define the $\bZ$-module quotient
\[ M \otimes_A^\beta N \coleq (M \otimes_\beta N) / J \]
which is a vector space because $\bK$ is contained in the center of $A$. It is endowed with the quotient topology, which is locally convex and Hausdorff. Denoting by $q$ the canonical mapping into the quotient we obtain a bilinear map
\[ \otimes_A^\beta \coleq q \circ \otimes\colon M \times N \to M \otimes_A^\beta N. \]

Similarly, if $A,M,N$ are taken to be locally convex instead of bounded, we denote the resulting space by $M \otimes_A^\pi N$ with corresponding mapping $\otimes_A^\pi$:
\begin{gather*}
M \otimes_A^\pi N \coleq (M \otimes_\pi N) / J \\
\otimes_A^\pi \coleq q \circ \otimes\colon M \times N \to M \otimes_A^\pi N.
\end{gather*}

\begin{definition}We call $M \otimes_A^\beta N$ the \emph{bornological tensor product} and $M \otimes_A^\pi N$ the \emph{projective tensor product} of $M$ and $N$ over $A$.
\end{definition}

By \cite[13.5 Prop.\ 1 (b)]{Jarchow} $M \otimes_A^\beta N$ is bornological. These spaces have the following universal properties.

\begin{proposition}\label{tpiso}Let $M$ be a right module over an algebra $A$, $N$ a left module over $A$ and $E$ a locally convex space. If $M$, $N$, and $A$ are locally convex then:
\begin{enumerate}
 \item[(i)] Given a continuous $\bK$-linear mapping $g\colon M \otimes_A^\pi N \to E$ the mapping $f \coleq g \circ \otimes_A^\pi$ is continuous, $\bK$-bilinear and $A$-balanced.
 \item[(ii)] Given a continuous $A$-balanced $\bK$-bilinear mapping $f\colon M \times N \to E$ there exists a unique continuous $\bK$-linear mapping $g\colon M \otimes_A^\pi N \to E$ such that $f = g \circ \otimes_A^\pi$.
\end{enumerate}
This gives an algebraic vector space isomorphism 
\begin{equation}\label{oktober_beta}
\Linc(M \otimes_A^\pi N, E) \cong \Lin^{A,c}(M,N;E).
\end{equation}
If $M$, $N$, and $A$ are bounded then:
\begin{enumerate}
 \item[(iii)] Given a bounded $\bK$-linear mapping $g\colon M \otimes_A^\beta N \to E$ the mapping $f \coleq g \circ \otimes_A^\beta$ is bounded, $\bK$-bilinear, and $A$-balanced.
 \item[(iv)] Given a bounded $A$-balanced $\bK$-bilinear mapping $f\colon M \times N \to E$ there exists a unique bounded $\bK$-linear mapping $g\colon M \otimes_A^\beta N \to E$ such that $f = g \circ \otimes_A^\beta$.
\end{enumerate}
This gives bornological vector space isomorphisms
\begin{equation}\label{oktober_alpha}
\begin{aligned}
 \Linb(M \otimes_A^\beta N, E)  & \cong \Lin^{A,b}(M,N;E) \\
& \cong \Linb_A(M, \Linb(N, E)) \cong \Linb_A(N, \Linb(M, E))
\end{aligned}
\end{equation}
\end{proposition}
\begin{proof}
(i) and (iii) are trivial. For (ii) and (iv) we obtain from \eqref{bornisolb} and \eqref{bornisolc} a unique mapping $\tilde f$ in $\Linc(M \otimes_\pi N, G)$ or $\Linb(M \otimes_\beta N, G)$ such that $f = \tilde f \circ \otimes$. Noting that $M \otimes_\beta N$ is bornological, $\tilde f$ is continuous in both cases and thus vanishes on $J$, whence there exists a unique linear mapping $g$ from $M \otimes_A^\pi N$ (or $M \otimes_A^\beta N$) into $E$ such that $f = g \circ q \circ \otimes$, which means $g \circ \otimes_A^\pi$ or $f = g \circ \otimes_A^\beta$, respectively. Clearly $g$ is continuous (bounded) by definition. It is furthermore easily verified that the correspondence $f \leftrightsquigarrow g$ is a vector space isomorphism. Finally, the isomorphisms of \eqref{bornisolb} are easily seen to restrict to bornological isomorphisms
\begin{align*}
\Lin^{A,b}(M,N;E) & \cong \{\,T \in \Linb(M \otimes_\beta N, E): J_0 \subseteq \ker T\,\} \\
& \cong \Linb_A(M, \Linb(N, E)) \cong \Linb_A(N, \Linb(M, E))
\end{align*}
Together with Lemma \ref{bornquot} this gives the result (we can replace $J_0$ by $J$ because $M \otimes_\beta N$ is bornological).

\end{proof}

\begin{lemma}\label{bornquot}Let $E$ be a bornological locally convex space, $N$ a closed subspace of $E$ and $F$ an arbitrary locally convex space. Then there is a bornological isomorphism
\[ \Linb(E/N, F) \cong \{ T \in \Linb(E,F): N \subseteq \ker T \} \]
where the latter space is equipped with the subspace topology.
\end{lemma}

\begin{proof}Denote by $p\colon E \to E/N$ the canonical projection. As to the algebraic part, for $\tilde T \in \Lin^b(E/N, F)$ the mapping $T \coleq \tilde T \circ p$ is in $\Linb(E,F)$ and vanishes on $N$; conversely, given such $T$ there exists a unique linear mapping $\tilde T$ such that $T = \tilde T \circ p$. Now $T$ is continuous (equivalently bounded) if and only if $\tilde T$ is (\cite[Proposition 4.6]{Treves}). The correspondences $T \leftrightsquigarrow \tilde T$ are inverse to each other and linear because the transpose $p^*$ of $p$ is linear.

For boundedness of $p^*$ let $\tilde B \subseteq L^b(E/N, F)$ be bounded and set $B \coleq p^*(\tilde B)$. Let $D \subseteq E$ be bounded and $V$ be a $0$-neighborhood in $F$. Then $\tilde D \coleq p(D)$ is bounded in $E/N$ so there exists $\lambda>0$ such that
\begin{align*}
 \tilde B &\subseteq \lambda \cdot\{\,\tilde T \in \Linb(E/N, F): \tilde T(\tilde D) \subseteq V \,\}
\intertext{and thus}
B &\subseteq \lambda \cdot \{\, p^*(\tilde T): \tilde T \in \Linb(E/N, F),\tilde T(\tilde D) \subseteq V \,\} \\
 &= \lambda \cdot \{\, T \in \Linb(E, F): N \subseteq \ker T, T(D) \subseteq V\, \}.
\end{align*}
which implies that $B$ is bounded. Conversely, let $B \subseteq \{T \in \Linb(E,F): N \subseteq \ker T \}$ be bounded and set $\tilde B \coleq (p^*)^{-1}(B) \subseteq \Linb(E/N, F)$. Let $\tilde D \subseteq E/N$ be bounded and $V$ a $0$-neighborhood in $F$. Because the images of bounded subsets of $E$ form a basis of the bornology
of $E/N$ (\cite[2:7]{henri}) there exists a bounded set $D \subseteq E$ such that $\tilde D \subseteq p(D)$. By assumption there is $\lambda>0$ such that
\begin{align*}
 B &\subseteq \lambda \cdot \{ T \in \Linb(E,F): N \subseteq \ker T, T(D) \subseteq V \} \\
\intertext{and thus}
\tilde B &\subseteq \lambda \cdot \{ \, (p^*)^{-1}(T): T \in \Linb(E,F), N \subseteq \ker T, T(D) \subseteq V\,\} \\
& \subseteq \lambda \cdot \{\, \tilde T \in \Linb(E/N, F): \tilde T(\tilde D) \subseteq V \,\}.\qedhere
\end{align*}
\end{proof}

We remark that the tensor product can also be constructed in a different way. Remember that as $\bK$ is in the center of $A$ $E \otimes_A F$ has a canonical vector space structure (\cite[II \S 3.6 Remark (2)]{Bourbaki}). For the following Lemma the separated vector bornology associated with a vector bornology is defined as the quotient bornology with respect to the Mackey closure $\overline{\{0\}}^\mathrm{b}$ of $\{0\}$ (\cite[2:12 Definition (2)]{henri}).

\begin{lemma}\label{twoway}Let $M$ be a right module and $N$ a left module over an algebra $A$. Then
\begin{enumerate}
 \item[(i)] If $M$, $N$ and $A$ are locally convex the Hausdorff space associated with the algebraic tensor product $M \otimes_A N$ endowed with the inductive locally convex topology with respect to the canonical mapping $\otimes\colon M \times N \to M \otimes_A N$ is homeomorphic to $M \otimes_A^\pi N$.
\item[(ii)] If $M$, $N$, and $A$ are bounded the separated bornological vector space associated with the algebraic tensor product $M \otimes_A N$ endowed with the final convex bornology with respect to the canonical mapping $\otimes$ is bornologically isomorphic to $M \otimes_A^\beta N$.
\end{enumerate}
\end{lemma}

\begin{proof}
(i)
Let $p\colon M \otimes_A N \to (M \otimes_A N) / \overline{\{0 \}}$ denote the canonical projection onto the quotient space.
\[
 \xymatrix{
**[l]M \times N \ar[d]_{\otimes_A^\pi} \ar[r]^\otimes & **[r] M \otimes_A N \ar[d]^p \ar[dl]_{\tilde g}\\
**[l]M \otimes_A^\pi N \ar@<0.5ex>[r]^-f & **[r] \ar@<0.5ex>[l]^-g (M \otimes_A N) / \overline{\{ 0 \}} \\
}
\]
Let $f$ be the continuous linear mapping induced by the continuous $A$-balanced $\bK$-bilinear mapping $p \circ \otimes$. $\otimes_A^\pi$ induces a continuous linear mapping $\tilde g$, which is continuous (and thus its kernel contains the closure of $\{0\}$); hence there exists a linear continuous mapping $g$ with $g \circ p = \tilde g$. In order to see that $f$ and $g$ are inverse to each other, we note that as $p$ is surjective and the images of $\otimes$ resp.\ $\otimes_A^\pi$ generate $M \otimes_A N$ resp.\ $M \otimes_A^\pi N$ it suffices to have the identities
\begin{gather*}
 f \circ g \circ p \circ \otimes = f \circ \otimes_A^\pi = p \circ \otimes \\
g \circ f \circ \otimes_A^\pi = g \circ p \circ \otimes = \otimes_A^\pi.
\end{gather*}

(ii) Replace $\overline{\{0\}}$ by $\overline{\{0\}}^\textrm{b}$, $\otimes_A^\pi$ by $\otimes_A^\beta$ and ``continuous'' by ``bounded'' in the proof of (i). Apply Lemma \ref{hausdorffsepp} to see that $M \otimes_A^\beta N$ is a separated bornological space and use \cite[2:12 Proposition (2)]{henri} for obtaining $g$.
\end{proof}

If $A$ is commutative $M \otimes_A^\beta N$ and $M \otimes_A^\pi N$ have a canonical $A$-module structure given by $a \cdot (m \otimes_A^\pi n) \coleq (ma) \otimes_A^\pi n$, which is bounded or locally convex, respectively:

\begin{proposition}If $A$ is commutative then $M \otimes_A^\beta N$ is a bounded $A$-module and $M \otimes_A^\pi N$ a locally convex $A$-module.
\end{proposition}
\begin{proof}
 For the bounded case see \cite[5.21]{KM}, for the continuous case \cite[Proposition II.2.2]{capelle} or \cite[Section 6.2]{Navarro}. 
\end{proof}

\begin{corollary}\label{bornisoalin}
If $A$ is commutative then the isomorphisms \eqref{oktober_beta} and \eqref{oktober_alpha} induce, for any locally convex $A$-modules $M$, $N$ and $P$, an algebraic isomorphism
\[ \Linc_A(M,N; P) \cong \Linc_A(M \otimes_A^\pi N, P). \]
and, for bounded $A$-modules $M$, $N$, and $P$, bornological isomorphisms
\begin{align*}
 \Linb_A(M \otimes_A^\beta N, P) & \cong \Linb_A(M,N; P) \\
 &\cong \Linb_A(M, \Linb_A(N, P)) \cong \Linb_A(N, \Linb_A(M, P)).
\end{align*}
\end{corollary}

\begin{proposition}\label{tpfcont}Let $f\colon M \to M'$ and $g\colon N \to N'$ be bounded (continuous) $A$-linear mappings between bounded (locally convex) $A$-modules. Then $f \otimes g$ is bounded (continuous).
\end{proposition}
\begin{proof}Because the mapping $(m,n) \mapsto f(m) \otimes g(n)$ from $M \times N$ into $M' \otimes_A N'$ is $A$-bilinear and bounded (continuous) the corresponding $A$-linear mapping $f \otimes g$ from $M \otimes_A^\beta N$ to $M' \otimes_A^\beta N'$ (from $M \otimes_A^\pi N$ to $M' \otimes_A^\pi N')$ such that $(f \otimes g) (m \otimes n) = f(m) \otimes g(n)$ is bounded (continuous).
\end{proof}

The following is an analogue of \cite[Proposition 5.8]{KM}, telling us when the bounded and projective tensor product are identical.

\begin{lemma}\label{topgleich}If every bounded bilinear mapping on $M \times N$ into an arbitrary locally convex space is continuous then $M \otimes_A^\pi N = M \otimes_A^\beta N$.
\end{lemma}
\begin{proof}
By construction, the topology of $M \otimes_A^\beta N$ is finer than the topology of $M \otimes_A^\pi N$: the identity $M \otimes_A^\beta N \to M \otimes_A^\pi N$ is continuous if and only if it is bounded (as $M \otimes_A^\beta N$ is bornological), which is the case if and only if $\id \circ \otimes_A^\beta = \otimes_A^\pi$ is bounded, but this mapping is even continuous.

\[
 \xymatrix{
M \times N \ar[d]_{\otimes_A^\beta} \ar[dr]^{\otimes_A^\pi} \\
M \otimes_A^\beta N \ar@<0.5ex>[r]^{\id} & M \otimes_A^\pi N \ar@<0.5ex>[l]^{\id}
}
\]

Conversely, the identity $M \otimes_A^\pi N \to M \otimes_A^\beta N$ is continuous if and only if $\id \circ \otimes_A^\pi = \otimes_A^\beta$ is continuous, which is the case by assumption because it is bounded and bilinear.
\end{proof}

By \cite[Proposition 5.8]{KM} the assumption of Lemma \ref{topgleich} is satisfied if $M$ and $N$ are metrizable, or if $M$ and $N$ are bornological and every separately continuous bilinear mapping on $E \times F$ is continuous.

\section{Topology on section spaces}\label{sec_sectop}

We will now define the standard topology on the space of sections of a finite dimensional vector bundle which turns it into a Fr\'echet space. In the following, the notions of a base of continuous seminorms and a family of seminorms defining the topology is as in \cite[Chapter 7]{Treves}.

For any open subset $\Omega$ of $\bR^n$ or of a manifold $M$ we call a sequence of sets $K_i \subseteq \Omega$ such that $\Omega = \bigcup_{i=1}^{\infty}K_i$ and each $K_i$ is compact and contained in the interior of $K_{i+1}$ a \emph{compact exhaustion} of $\Omega$.

Let $\Omega \subseteq \bR^n$ be open and $(\bE, \norm{\ })$ a Banach space. The space $C^\infty(\Omega, \bE)$ of all smooth functions from $\Omega$ into $\bE$ has the usual Fr\'echet structure (\cite[Chapter 40]{Treves}): defining seminorms $\fp_{K,k}$ (for $K \subseteq \Omega$ compact and $k \in \bN_0$) on $C^\infty(\Omega, \bE)$ by
\[ \fp_{K,k}(f) \coleq \max_{\abso{\alpha}\le k, x \in K}\norm{\pd^\alpha f(x)} \]
the topology of $C^\infty(\Omega, \bE)$ has as basis of continuous seminorms
the family $\{\, \fp_{K_n, k}\ |\ n \in \bN, k \in \bN_0\,\}$ where $(K_n)_n$ is any compact exhaustion of $\Omega$. This topology evidently does not depend on the choice of the compact exhaustion.

Now let $M$ be an $n$-dimensional manifold with atlas $\{(U_i, \varphi_i)\}_i$ and $\pi\colon E \to M$ a vector bundle whose typical fiber is an $m$-dimensional
Banach space $\bE$. Let $\{(V_j, \tau_j)\}_j$ be a trivializing covering of $E$. For any $i$ and $j$ a section $s \in \Gamma(E)$ has local representation
\[ s_{U_i, V_j} \coleq \pr_2 \circ \tau_j \circ s|_{U_i \cap V_j} \circ (\varphi_i|_{U_i \cap V_j})^{-1} \in C^\infty(\varphi_i(U_i \cap V_j), \bE) \]
where $\pr_2$ is the projection on the second component. This is illustrated by the following diagram.
\[
 \xymatrix{
\pi^{-1}(U_i \cap V_j) \ar[dr]^{\tau_j} & \\
U_i \cap V_j \ar[u]^{s_{U_i \cap V_j}} \ar[r]& U_i \cap V_j \times \bE \ar[d]^{\pr_2} \\
\varphi_i(U_i \cap V_j) \ar[u]^{\varphi_i^{-1}} \ar[r]^-{s_{U_i, V_j}}& \bE
}
\]

$\Gamma(E)$ then is endowed with the (locally convex) projective topology $\sT_E$ defined by the linear mappings
\[
\Gamma(E) \ni s \mapsto s_{U_i,V_j} \in C^\infty(\varphi_i(U_i \cap V_j), \bE)
\]
for all $i$ and $j$, which is complete by \cite[II 5.3]{Schaefer}. For a description by seminorms we 
set $\fp_{U_i, V_j, K, k}(s) \coleq \fp_{\varphi_i(K),k}(s_{U_i, V_j})$ for $s \in \Gamma(E)$. The topology $\sT_E$ has as basis of continuous seminorms the family $\mathfrak{P}_E$ given by all $\fp_{U_i, V_j, K_n, k}$ for $k \in \bN_0$, $(K_n)_n$ a compact exhaustion of $U_i \cap V_j$, and all $i,j,n,k$.
Because for each $s \in \Gamma(E) \setminus\{0\}$ there is some $\fp \in \mathfrak{P}_E$ such that $\fp(s)>0$, $\sT_E$ is Hausdorff.

\begin{proposition}\label{indepatlas}$\sT_E$ is independent of the atlas, the trivializing covering and the compact exhaustions. 
\end{proposition}
\begin{proof}
 Let $M$ have atlases $\{(U_i, \varphi_i)\}_i$ and $\{(\tilde U_k, \tilde \varphi_k)\}_k$ and let $E$ have trivializing coverings $\{(V_j, \tau_j)\}_j$ and $\{(\tilde V_l, \tilde \tau_l)\}_l$. This gives rise to topologies $\sT_E$ and $\tilde \sT_E$ on $\Gamma(E)$. For continuity of the identity mapping $[\Gamma(E), \sT_E] \to [\Gamma(E), \tilde \sT_E]$ it suffices to show that for all $k,l$, every compact exhaustion $(\tilde K_m)_m$ of $\tilde U_k \cap \tilde V_l$, and all $m$, $p$ there 
is a continuous seminorm $\fp$ of $(\Gamma(E), \sT_E)$ such that
\begin{equation}\label{eins}
\fp_{\tilde U_k,\tilde V_l, \tilde K_m, p}(s) \le \fp(s).
\end{equation}

First, we show that we can assume that $\tilde K_m$ is contained in some $U_i \cap V_j$. As the open sets $U_i \cap V_j$ form an open cover of $M$ we can write $\tilde K_m$ as the disjoint union of finitely many $\tilde K_m^{a,b} \csub U_{i(a)} \cap V_{j(b)} \cap \tilde U_k \cap \tilde V_l$. Assuming that \eqref{eins} holds in this case there are continuous seminorms $\fp_{a,b}$ of $\sT_E$ such that
\[ \fp_{\tilde U_k,\tilde V_l, \tilde K_m^{a,b}, p}(s) \le \fp_{a,b}(s) \]
for all $a,b$. We take the maximum over all $a,b$ on both sides and obtain $\fp_{\tilde U_k, \tilde V_l, \tilde K_m, p}$ on the left side and a continuous seminorm $\fp$ on the right side. Thus we may assume that $K \coleq \tilde K_m \csub U_i \cap V_j \cap \tilde U_k \cap \tilde V_l$ for some $i,j,k,l$.
The left side of \eqref{eins} is then given by
\[ \max_{\substack{\abso{\alpha}\le p\\x \in \tilde \varphi_k(K)}}\norm{\pd^\alpha s_{\tilde U_k,\tilde V_l}(x)}. \]
For $x \in \tilde \varphi_k(K)$ we then write
\begin{align*}
s_{\tilde U_k, \tilde V_l}(x) &= \pr_2 \circ \tilde\tau_l \circ \tau_j^{-1}(\tilde \varphi_k^{-1}(x), s_{U_i, V_j} \circ \varphi_i \circ \tilde \varphi_k^{-1}(x)) \\
&= (\tilde \tau_l \circ \tau_j^{-1})_{\tilde\varphi_k^{-1}(x)}(s_{U_i, V_j} \circ \varphi_i \circ \tilde \varphi_k^{-1}(x))
\end{align*}
where the transition mapping $x \mapsto (\tilde \tau_l \circ \tau_j^{-1})_x$ is a smooth function from $V_j \cap \tilde V_l$ to $\Linc(\bE, \bE)$.
By the product rule we obtain for $\pd^\alpha s_{\tilde U_k, \tilde V_l}(x)$ terms of the form
\[ \pd^\beta [x \mapsto (\tilde \tau_l \circ \tau_j^{-1})_{\tilde \varphi_k^{-1}(x)}] \cdot \pd^\gamma [ x \mapsto  s_{U_i, V_j} (\varphi_i \circ \tilde \varphi_k^{-1}(x))] \]
for some multi-indices $\beta,\gamma$. Taking the maximum over $x \in \tilde \varphi_k(K)$, the first factor gives a constant and the second factor gives a sum of terms of the form
\[ \max_{x \in \varphi_i(K)}\norm{\pd^{\gamma'} s_{U_i, V_j}(x)} \le \fp_{U_i, V_j, K, \abso{\gamma'}}(s)\]
for some $\gamma'$. Altogether, these terms give a continuous seminorm of $\sT_E$, whence the identity mapping from $[\Gamma(E), \tilde\sT_E] \to [\Gamma(E), \sT_E]$ is continuous. By symmetry we have a homeomorphism.
\end{proof}

Because the trivializing covering of $E$ and the atlas of $M$ can be assumed to be countable (\cite[1.4.8]{brickell}) $\sT_E$ is determined by a countable family of seminorms. Therefore, $[\Gamma(E), \sT_E]$ as well as its closed subspace $\Gamma_{c,L}(E)$ for a compact set $L \csub M$ with the subspace topology are Fr\'echet spaces.

In order to turn $\Gamma_c(E)$ into a complete topological space we endow it with the strict inductive limit topology of a suitable sequence of Fr\'echet subspaces, which by \cite[II 6.6]{Schaefer} then is complete. As $M$ is $\sigma$-compact we obtain an (LF)-space $\Gamma_c(E) = \varinjlim \Gamma_{c,L}(E)$ where $L$ ranges through a compact exhaustion of $M$.

For the particular case of $\CinfM$ we abbreviate
$\fp_{U_i,K,k} \coleq \fp_{U_i, U_i, K, k}$.
Then we obtain a basis of continuous seminorms
\[ \mathfrak{P}_M \coleq \{\, \fp_{U_i,K_n^i,k}\ |\ k \in \bN_0, n \in \bN, i\,\} \]
where $\{K^i_n\}_n$ is a compact exhaustion of $\varphi_i(U_i)$.

We now state simple lemmata (proof omitted) about continuity of bilinear mappings as determined by seminorms.

\begin{lemma}\label{bilinc}Let $E$, $F$ and $G$ be topological vector spaces with locally convex topology. A bilinear mapping $f\colon E \times F \to G$ is continuous if and only if for each continuous seminorm $r$ on $G$ there are continuous seminorms $p$ on $E$ and $q$ on $F$ such that for all $x \in E$ and $y \in F$ we have $r(f(x,y)) \le p(x) q(x)$.
%
%

If $\cP_E$ resp.~$\cP_F$ are bases of continuous seminorms on $E$ resp.~$F$ and $\cS_G$ a family of seminorms defining the topology of $G$ then a bilinear mapping $f\colon E \times F \to G$ is continuous if and only if for each $r \in \cS_G$ there are seminorms $p \in \cP_E$ and $q \in \cP_F$ and a constant $C>0$ such that $r(f(x,y)) \le C p(x) q(x)$ for all $x \in E$, $y \in F$.
\end{lemma}

\begin{lemma}
\begin{enumerate}
\item[(i)]$\CinfM$ is a locally convex unital commutative associative algebra.
\item[(ii)]For any vector bundle $E$ the space of sections $\Gamma(E)$ is a Hausdorff locally convex module over $\CinfM$.
\end{enumerate}
\end{lemma}

\begin{proof}
We will verify continuity of the respective multiplication mappings, the rest being immediately clear from the definitions. Let $\{(U_i, \varphi_i)\}_i$ be an atlas of $M$ and $\{(U_i, \tau_i)\}_i$ a trivializing covering of $\Gamma(E)$ -- by Proposition \ref{indepatlas} we can always intersect the domains of the atlas and the trivializing covering in order to have them in this form. By the product rule for differentiation we obtain
\begin{align*}
\fp_{i,K,k}(fg) &\le C \fp_{i,K,k}(f) \cdot \fp_{i, K, k}(g)\textrm{ and}\\
\fp_{i,K,k}(fs) &\le C \fp_{i,K,k}(f) \cdot \fp_{U_i, U_i, K, k}(s)
\end{align*}
for all $K \csub U_i$, $k \in \bN_0$, $f,g \in \CinfM$, $s \in \Gamma(E)$, and some constant $C>0$.
\end{proof}

\begin{lemma}\label{dualcont}Given a trivial vector bundle $E$ and a basis $\{ b_1,\dotsc, b_n\}$ of $\Gamma(E)$ the corresponding dual basis $\{ b_1^*,\dotsc,b_n^* \}$ consists of elements of $\Linc_\CinfM(\Gamma(E), \CinfM)$.
\end{lemma}
\begin{proof}Let $\tau\colon E \to M \times \bR^n$ be trivializing. For the basis $\alpha_i(x) \coleq \tau^{-1}(x, e_i)$ where $\{e_1,\dotsc,e_n\}$ is the canonical basis of $\bR^n$ the result is clear, as the dual basis is then given by $\alpha_i^*(s)(x) = \pr_i \circ \pr_2 \circ \tau \circ s$. For an arbitrary basis $\{b_1,\dotsc,b_n\}$ we know that $b_i^* = \sum a_i^j \alpha_j^*$ for some $a_i^j \in \CinfM$. Because for $f \in \CinfM$ the mapping $s \mapsto (f \alpha_j^*)(s) = f \cdot \alpha_j(s)$ is the composition of $\alpha_j$ and multiplication with $f$, both continuous, $b_i^*$ is the sum of continuous mappings.
\end{proof}

We recall the following basic facts about products and direct sums of topological vector spaces. Let $(M_i)_i$ be a family of topological vector spaces. The product $\prod_i M_i$ carries the projective topology defined by the canonical projections $\pi_i$ and the external direct sum $\bigoplus_i M_i$ the inductive linear topology with respect to the canonical injections, which makes them topological vector spaces. If all $M_i$ are locally convex $A$-modules $\prod_i M_i$ is a locally convex $A$-module: denoting the multiplication mappings by $m\colon A \times \prod_i M_i \to \prod_i M_i$ resp.\ $m_i\colon A \times M_i \to M_i$, $m$ is continuous because $\pi_i \circ m = m_i \circ (\id \times \pi_i)$ is continuous for each $i$. For finitely many factors (which is all we will need) $\bigoplus_i M_i = \prod_i M_i$ topologically.

We will now establish preliminaries required for a topological version of the isomorphism $\Gamma(E \otimes F) \cong \Gamma(E) \otimes_{\CinfM} \Gamma(F)$; we need explicit expressions as well as continuity of some canonical isomorphisms. 

\begin{proposition}\label{secdsum}Given vector bundles $E_1, \dotsc, E_n$ the canonical isomorphism of $\CinfM$-modules
\[ \Gamma(\bigoplus_{\mathclap{i=1\dotsc n}} E_i) \cong \bigoplus_{\mathclap{i=1\dotsc n}}\Gamma(E_i) \]
is a homeomorphism.
\end{proposition}
\begin{proof}
For each $x \in M$ let $\iota_j\colon E_{jx} \to \bigoplus_{i=1\dotsc n} E_{ix}$ denote the canonical injection of the fiber $E_{jx}$ into the direct sum and $\pi_j\colon \bigoplus_{i=1\dotsc n}E_{ix} \to E_{jx}$ the canonical projection. Define injections and projections, respectively, by
\begin{gather*}
\tilde \iota_j\colon \Gamma(E_j) \to \Gamma(\bigoplus_{\mathclap{i=1\dotsc n}} E_i),\quad (\tilde \iota_j s_j)(x) \coleq \iota_j(s_j(x))\quad\text{for }s_j \in \Gamma(E_j),\\
\tilde \pi_j\colon \Gamma(\bigoplus_{\mathclap{i=1\dotsc n}}E_i) \to \Gamma(E_j),\quad (\tilde \pi_j s)(x) \coleq  \pi_j(s(x))\quad\text{for }s \in \Gamma(\bigoplus_{\mathclap{i=1\dotsc n}}E_i).
\end{gather*}

We have to verify that the images of $\tilde \iota_j$ and $\tilde \pi_j$ are indeed smooth sections. Let $\{U_l, \varphi_l\}_l$ be an atlas of $M$ and $\{(V^j_{k_j}, \tau^j_{k_j})\}_{k_j}$ trivializing coverings of $E_j$, then $\bigoplus_{i=1\dotsc n}E_i$ has trivializing covering
\[ \{ ( \bigcap_{\mathclap{j=1\dotsc n}}V^j_{k_j}, \sigma_{k_1,\dotsc,k_n})\}_{k_1,\dotsc,k_n} \]
where $(\sigma_{k_1,\dotsc,k_n})_x(t) \coleq (x, (\pr_2 \tau^1_{k_1}\pi_1 t,\dotsc, \pr_2 \tau^n_{k_n}\pi_n t))$ for $t \in \bigoplus_{j=1\dotsc n} E_{jx}$ and $x \in \bigcap_{j=1\dotsc n}V^j_{k_j}$. First, let $s_j \in \Gamma(E_j)$; then on each chart domain $U_l \cap V^1_{k_1} \cap \dotsc \cap V^n_{k_n}$, $\pr_2 \circ \sigma_{k_1,\dotsc,k_n} \circ \tilde\iota_j(s_j) \circ \varphi_l^{-1}$ is smooth because its only nonzero component is $\pr_2 \circ \tau_{k_j}^j \circ s_j \circ \varphi_l^{-1}$ which is smooth by assumption. Conversely, let $s \in \Gamma(\bigoplus_{i=1 \dotsc n} E_i)$. Then on each chart domain as above $\pr_2 \circ \tau^j_{k_j} \circ \tilde \pi_j(s) \circ \varphi_l^{-1} = \pr_2 \circ \tau^j_{k_j} \circ \pi_j  \circ s \circ \varphi_l^{-1} = \pr_j \circ \pr_2 \circ \sigma_{k_1,\dotsc,k_n} \circ s \circ \varphi_l^{-1}$ is smooth. Finally, $\tilde \pi_k \circ \tilde \iota_j = \id$ for $k=j$ and $0$ otherwise; as $\sum_j \tilde \iota_j \circ \tilde \pi_j(s) = s$, $\Gamma(\bigoplus_{j=1\dotsc n} E_j)$ is a direct product for the family of $\CinfM$-modules $(\Gamma(E_j))_j$ (\cite[Theorem 6.7]{blyth}) and algebraically isomorphic to $\bigoplus_{j=1\dotsc n}\Gamma(E_j)$. The isomorphism $\psi\colon \Gamma(\bigoplus_{j=1\dotsc n}E_j) \to \bigoplus_{j=1\dotsc n}\Gamma(E_j)$ is given by
\begin{align*}
\psi(s) &= (\tilde \pi_1(s),\dotsc,\tilde \pi_n(s)) \textrm{ and}\\
\psi^{-1}(s_1,\dotsc,s_n) &= \tilde\iota_1(s_1) + \dotsc + \tilde \iota_n(s_n).
\end{align*}
Continuity of $\tilde \pi_j$ and $\tilde \iota_j$ is easily seen from the respective seminorms, which implies continuity of $\psi$ and $\psi^{-1}$.
\end{proof}

\begin{lemma}\label{vbdsum} For vector bundles $E_1,\dots,E_n$ and $F_1,\dotsc,F_m$ over $M$ we have a canonical vector bundle isomorphism
\[ (\bigoplus_{\mathclap{i=1\dotsc n}}E_i) \otimes (\bigoplus_{\mathclap{j=1\dotsc m}} F_j) \cong \bigoplus_{\mathclap{\substack{i=1 \dotsc n\\j=1 \dotsc m}}} (E_i \otimes F_j) \]
\end{lemma}
\begin{proof}
Evidently the fiberwise defined map
\[ (v_1,\dotsc, v_n) \otimes (w_1,\dotsc,w_m) \mapsto (v_1 \otimes w_1, \dotsc, v_n \otimes w_m ) \]
(where $v_i \in E_{ix}$ and $w_j \in F_{jx}$ for all $i,j$ and fixed $x$) is a vector bundle isomorphism over the identity. Its inverse is induced by the mappings
\[ e_i \otimes f_j \mapsto \iota_i e_i \otimes \iota_j f_j\qquad (e_i \in E_{ix}, f_j \in F_{jx}) \]
for all $i,j$, where $\iota_i$ and $\iota_j$ are the canonical injections $E_{ix} \to \bigoplus_i E_{ix}$ and $F_{jx} \to \bigoplus_j F_{jx}$, respectively.
\end{proof}

\begin{lemma}\label{vbmodiso} For isomorphic vector bundles $E \cong F$ the canonical $\CinfM$-module isomorphism $\Gamma(E) \cong \Gamma(F)$ is a homeomorphism.
\end{lemma}
\begin{proof}
If $(f, f_0)$ is the vector bundle isomorphism from $E$ to $F$ the isomorphism $\Gamma(E) \to \Gamma(F)$ is given by $s \mapsto f \circ s \circ f_0^{-1}$. It is readily verified using the respective seminorms that this assignment and its inverse are continuous.
\end{proof}


\begin{lemma}\label{plustensiso}Let $A$ be a locally convex algebra, $M_i$ ($i=1,\dotsc,n)$ locally convex right $A$-modules, and $N_j$ ($j=1,\dotsc,m)$ locally convex left $A$-modules. Then the canonical vector space isomorphism
\[ (\bigoplus_{\mathclap{i=1\dotsc n}} M_i ) \otimes (\bigoplus_{\mathclap{j=1\dotsc m}} N_j ) \cong \bigoplus_{\mathclap{\substack{i=1\dotsc n\\j=1\dotsc m}}} (M_i \otimes N_j) \]
induces isomorphisms of locally convex spaces
\begin{gather*}
(\bigoplus_{\mathclap{i=1\dotsc n}} M_i ) \otimes_\pi (\bigoplus_{\mathclap{j=1\dotsc m}} N_j ) \cong \bigoplus_{\mathclap{\substack{i=1\dotsc n\\j=1\dotsc m}}} (M_i \otimes_\pi N_j) \\
(\bigoplus_{\mathclap{i=1\dotsc n}} M_i ) \otimes_A^\pi (\bigoplus_{\mathclap{j=1\dotsc m}} N_j ) \cong \bigoplus_{\mathclap{\substack{i=1\dotsc n\\j=1\dotsc m}}} (M_i \otimes_A^\pi N_j).
\end{gather*}
If $A$ is commutative the last one is an isomorphism of $A$-modules.
\end{lemma}
\begin{proof}
By \cite[II \S3.7 Proposition 7]{Bourbaki} the mapping
\begin{align*}
 g\colon (\bigoplus_{\mathclap{i=1\dotsc n}} M_i ) \otimes (\bigoplus_{\mathclap{j=1\dotsc m}} N_j) &\to \bigoplus_{\mathclap{\substack{i=1\dotsc n\\j=1\dotsc m}}} (M_i \otimes N_j) \\
(m_i)_i \otimes (n_j)_j &\mapsto (m_i \otimes n_j)_{i,j}
\end{align*}
is a vector space isomorphism. Its inverse $h$ is induced by the mappings $h_{ij} \coleq \iota_i \otimes \iota_j$, where $\iota_i\colon M_i \to \bigoplus M_i$ and $\iota_j\colon N_j \to \bigoplus N_j$ are the canonical injections. This means that $h$ is given by $\sum_{ij} h_{ij} \circ \pr_{ij}$ where $\pr_{ij}$ is the canonical projection $\bigoplus_{ij} (M_i \otimes N_j) \to M_i \otimes N_j$.

Define $J_0$ as the sub-$\bZ$-module of $(\bigoplus M_i) \otimes (\bigoplus N_h)$ generated by all elements of the form $(m_i)_i a \otimes (n_j)_j - (m_i)_i \otimes a (n_j)_j$, and $J_{ij}$ as the sub-$\bZ$-module of $M_i \otimes N_j$ generated by all elements of the form $m_i a \otimes n_j - m_i \otimes a n_j$ with $m_i \in M_i$, $n_i \in N_i$, and $a \in A$. As $\bK$ is in the center of $A$ these are vector subspaces. By \cite[II \S 1.6]{Bourbaki} there is a canonical isomorphism of vector spaces
\[ f\colon \bigoplus_{i,j} \frac{M_i \otimes_\pi N_j}{\overline{J_{ij}}} \to \frac{\bigoplus_{i,j} (M_i \otimes_\pi N_j)}{\bigoplus_{i,j}\overline{J_{ij}}} \]
induced by the mappings $f_{ij} (m_i \otimes n_j + \overline{J_{ij}}) \coleq \iota(m_i \otimes n_j) + \bigoplus_{k,l}\overline{J_{kl}}$. Thus we obtain the following commutative diagram.

\[
 \xymatrix{
\ar[d]_-q (\bigoplus_i M_i ) \otimes_\pi (\bigoplus_j N_j) \ar@<0.5ex>[r]^-g & \bigoplus_{i,j} (M_i \otimes_\pi N_j) \ar@<0.5ex>[l]^-h \ar[dr]^-{(p_{ij})_{i,j}} \ar[d]^-r & \\
\dfrac{(\bigoplus_i M_i) \otimes_\pi (\bigoplus_j N_j)}{\overline{J_0}} \ar@{.>}[r]^-\lambda&
\dfrac{\bigoplus_{i,j} (M_i \otimes_\pi N_j)}{\bigoplus_{i,j}\overline{J_{ij}}} &
\bigoplus_{i,j}\dfrac{M_i \otimes_\pi N_j}{\overline{J_{ij}}} \ar[l]_-f
}
\]

Here $q$, $r$, and $p_{ij}$ are the projections onto the respective quotient.

It is now easily seen that $g(J_0) = \bigoplus_{i,j} J_{ij}$ and if $g$ and $h$ are continuous $g(\overline{J_0}) = \bigoplus_{i,j}\overline{J_{ij}}$, which immediately implies that there exists a vector space isomorphism $\lambda$ as in the diagram. The claims then follow if we show $f$, $f^{-1}$, $g$ and $h$ to be continuous.

First, $g$ is induced by the $C^\infty(M)$-bilinear map
\begin{align*}
 \tilde g\colon (\bigoplus_i M_i) \times (\bigoplus_j N_j) &\to \bigoplus_{i,j} (M_i \otimes N_j) \\
((m_i)_i, (n_j)_j) &\mapsto (m_i \otimes n_j)_{i,j}
\end{align*}
and $g$ is continuous if and only if $\tilde g$ is. Because the target space has only finitely many summands continuity can be tested by composition with the projections $\pi_{ij}$ onto $M_i \otimes N_j$. As $\pi_{ij} \circ \tilde g = \otimes \circ (\pi_i \times \pi_j)$ is continuous $g$ is continuous.

Second, $h$ is continuous because the $h_{ij}$, which are the tensor product of continuous mappings, are so. Similarly, $f$ is continuous because $f \circ \iota_{ij} \circ p_{ij} = r \circ (\iota_i \circ \iota_j)$ is continuous, where $\iota_{ij}: (M_i \otimes_\pi N_j)/\overline{J_{ij}} \to \bigoplus_{i,j} (M_i \otimes_pi N_j)/\overline{J_{ij}}$ is the canonical inclusion.

Finally, $f^{-1}$ is continuous if and only if $f^{-1} \circ r = (p_{ij})_{i,j}$ is, which is the case because all $p_{ij}$ are continuous and we can test continuity into the finite direct sum by composition with the projections on each factor.
\end{proof}

Note that for infinitely many summands the previous lemma is false, in general (\cite[15.5, 1. Example]{Jarchow}).

\section{Tensor product of section spaces}\label{sec_iso}

\begin{theorem}\label{sectens}For any vector bundles $E$ and $F$ on $M$ the canonical $\CinfM$-module isomorphism $\Gamma(E) \otimes_{\CinfM} \Gamma(F) \cong \Gamma(E \otimes F)$ induces a homeomorphism $\Gamma(E) \otimes^\pi_\CinfM \Gamma(F) \cong \Gamma(E \otimes F)$.
\end{theorem}
\begin{proof}
Suppose first that $E$ and $F$ are trivial, then there are finite bases $\{\alpha_i\}_i$ and $\{\beta_j\}_j$ of $\Gamma(E)$ and $\Gamma(F)$, respectively. Clearly $E \otimes F$ then also is trivial and $\Gamma(E \otimes F)$ has a finite basis $\{\gamma_{ij}\}_{i,j}$. Explicitly these bases can be given as follows: suppose we have trivializing mappings $\tau\colon E \to M \times \bE$, $\sigma\colon F \to M \times \bF$ and $\mu\colon E \otimes F \to M \times (\bE \otimes \bF)$, with $\mu_x(v \otimes w) = (x, \pr_2 \circ \tau_x(v) \otimes \pr_2 \circ \sigma_x(w))$. Let $\{e_i\}_i$, $\{f_j\}_j$ be bases of $\bE$ resp.~$\bF$, which gives a basis $\{e_i \otimes f_j\}_{i,j}$ of $E \otimes F$. Then we set
\begin{equation*}
\begin{split}
\alpha_i(x) &\coleq \tau^{-1}(x, e_i), \\
\beta_j(x) &\coleq \sigma^{-1}(x, f_j),\textrm{ and} \\
\gamma_{ij}(x) &\coleq \mu^{-1}(x, e_i \otimes f_j) = \alpha_i(x) \otimes \beta_j(x).
\end{split}
\end{equation*}
Now $\{(\alpha_i,\beta_j)\}_{i,j}$ is a basis of $\Gamma(E) \times \Gamma(F)$.
There is a unique $\CinfM$-bilinear mapping $\tilde g\colon \Gamma(E) \times \Gamma(F) \to \Gamma(E \otimes F)$
such that $\tilde g(\alpha_i,  \beta_j)=\gamma_{ij}$ $\forall i,j$. Writing 
\[ \tilde g = \sum_{i,j} m \circ (\id \times m(\cdot, \gamma_{ij})) \circ (\alpha_i^* \times \beta_j^*) \]
where $m\colon \CinfM \times \Gamma(E \otimes F) \to \Gamma(E \otimes F)$ is module multiplication on $\Gamma(E \otimes F)$ and $\alpha_i^*$, $\beta_j^*$ are elements of the bases dual to $\{\alpha_i\}_i$ and $\{\beta_j\}_j$ (which are continuous by Lemma \ref{dualcont}) one sees that $\tilde g$ is continuous. Note that $g(t \otimes s)(x) = t(x) \otimes s(x)$ for $t \in \Gamma(E)$, $s \in \Gamma(E)$, and $x \in M$. By Corollary \ref{bornisoalin} $\tilde g$ induces a unique continuous $C^\infty(M)$-linear mapping $g\colon \Gamma(E) \otimes_{\CinfM}^\pi \Gamma(F) \to \Gamma(E \otimes F)$ such that $\tilde g = g \circ \otimes_\CinfM^\pi$.
\[
 \xymatrix{
\Gamma(E) \times \Gamma(F) \ar[r]^{\tilde g} \ar[d]_{\otimes_\CinfM^\pi} & \Gamma(E \otimes F) \ar@<0.5ex>[dl]^{h}\\
\Gamma(E) \otimes_\CinfM^\pi \Gamma(F) \ar@<0.5ex>[ru]^g & 
}
\]
For the inverse we define $h\colon \Gamma(E \otimes F) \to \Gamma(E) \otimes_\CinfM^\pi \Gamma(F)$ by $h(\gamma_{ij}) = \alpha_i \otimes_\CinfM^\pi \beta_j$, i.e., $h(s) = \sum_{i,j}\gamma_{ij}^*(s)\alpha_i \otimes_\CinfM^\pi \beta_j$ for $s \in \Gamma(E \otimes F$), which is continuous and $\CinfM$-linear.
Now it suffices to note that $g$ and $h$ are inverse to each other:
\begin{align*}
h(g(t \otimes_\CinfM^\pi u)) &= h(\tilde g(t^i\alpha_i, u^j\beta_j)) = h(t^iu^j\gamma_{ij}) = t^iu^j\alpha_i \otimes_\CinfM^\pi \beta_j\\
& = t^i\alpha_i \otimes_\CinfM^\pi u^j\beta_j = t \otimes_\CinfM^\pi u\textrm{ and}\\
g(h(s)) &= g(s^{ij}\alpha_i \otimes_\CinfM^\pi \beta_j) = s^{ij}\tilde g(\alpha_i, \beta_j) = s^{ij}\gamma_{ij} = s.
\end{align*}
Thus for trivial bundles we have established the $\CinfM$-module isomorphism and homeomorphism $\varphi_{E,F} \coleq h$,
\[ \varphi_{E,F}\colon \Gamma(E \otimes F) \to \Gamma(E) \otimes_\CinfM^\pi \Gamma(F). \]

 Now suppose that $E$ and $F$ are arbitrary vector bundles. Then by \cite[2.23]{GHV} there exist vector bundles $E'$ and $F'$ over $M$ such that $E \oplus E'$ and $F \oplus F'$ are trivial, giving an isomorphism $\varphi \coleq \varphi_{E \oplus E', F \oplus F'}$ as above:
\begin{equation}\label{isom1}
\Gamma((E \oplus E') \otimes (F \oplus F')) \cong \Gamma(E \oplus E') \otimes_\CinfM^\pi \Gamma(F \oplus F').
\end{equation}
We now distribute the direct sums on both sides and write down all isomorphisms involved. First, by Proposition \ref{secdsum} we have an isomorphism of $\CinfM$-modules and homeomorphism $\psi_{E,E'}\colon \Gamma(E \oplus E') \to \Gamma(E) \oplus \Gamma(E')$ given by
\begin{align*}
\psi_{E,E'}(s) & = [x \mapsto (\pr_1 \circ s(x), \pr_2 \circ s(x))]  = (\pr_1 \circ s, \pr_2 \circ s) \\
\psi_{E,E'}^{-1}(s_1, s_2) & = [x \mapsto (s_1(x), s_2(x))].
\end{align*}
As both $\psi \coleq \psi_{E,E'} \otimes_\CinfM^\pi \psi_{F,F'}$ and its inverse $\psi_{E,E'}^{-1} \otimes_\CinfM^\pi \psi_{F,F'}^{-1}$ are continuous (Proposition \ref{tpfcont}) we obtain an isomorphism of $\CinfM$-modules
\[ \psi: \Gamma(E \oplus E') \otimes_\CinfM^\pi \Gamma(F \oplus F') \to (\Gamma(E) \oplus \Gamma(E')) \otimes_\CinfM^\pi (\Gamma(F) \oplus \Gamma(F')) \] which also is a homeomorphism.
For the left hand side of \eqref{isom1}
we use the vector bundle isomorphism of Lemma \ref{vbdsum} given on each fiber by
\[ \kappa\colon (e, e') \otimes (f, f') \mapsto (e \otimes f, e \otimes f', e' \otimes f, e' \otimes f') \]
which by Lemma \ref{vbmodiso} gives a $\CinfM$-module isomorphism and homeomorphism $\lambda\colon s \mapsto \kappa \circ s$.

Let $\rho$ be the isomorphism from Lemma \ref{plustensiso} (denoted by $g$ in its proof). Explicitly, it maps
$(s,s') \otimes_\CinfM^\pi (t,t')$ to $(s \otimes_\CinfM^\pi t, s \otimes_\CinfM^\pi t', s' \otimes_\CinfM^\pi t, s' \otimes_\CinfM^\pi t')$.
Its inverse $\rho^{-1}$ is induced by the following mappings, all having image in the space $(\Gamma(E) \oplus \Gamma(E')) \otimes_\CinfM^\pi (\Gamma(F) \oplus \Gamma(F'))$:
\begin{align*}
\Gamma(E) \otimes_\CinfM^\pi \Gamma(F) \ni s_1 \otimes_\CinfM^\pi t_1 & \mapsto (s_1,0) \otimes_\CinfM^\pi (t_1,0), \\
\Gamma(E) \otimes_\CinfM^\pi \Gamma(F') \ni s_2 \otimes_\CinfM^\pi t_1' & \mapsto (s_2,0) \otimes_\CinfM^\pi (0,t_1'), \\
\Gamma(E') \otimes_\CinfM^\pi \Gamma(F) \ni s_1' \otimes_\CinfM^\pi t_2 & \mapsto (0,s_1') \otimes_\CinfM^\pi (t_2,0), \textrm{ and} \\
\Gamma(E') \otimes_\CinfM^\pi \Gamma(F') \ni s_2' \otimes_\CinfM^\pi t_2' & \mapsto (0,s_2') \otimes_\CinfM^\pi (0,t_2').
\end{align*}
This means that $\rho^{-1}(s_1 \otimes_\CinfM^\pi t_1, s_2 \otimes_\CinfM^\pi t_1', s_1' \otimes_\CinfM^\pi t_2, s_2' \otimes_\CinfM^\pi t_2')$ is given by
\begin{multline*}
(s_1, 0) \otimes_\CinfM^\pi (t_1, 0) + (s_2, 0 ) \otimes_\CinfM^\pi (0, t_1')\\
+ (0, s_1') \otimes_\CinfM^\pi (t_2, 0) + (0, s_2') \otimes_\CinfM^\pi (0, t_2').
\end{multline*}

The isomorphism $\Gamma(E) \otimes_\CinfM^\pi \Gamma(F) \cong \Gamma(E \otimes F)$ we are looking for will now be obtained as a component of $f \coleq \lambda \circ \varphi^{-1} \circ \psi^{-1} \circ \rho^{-1}$.
Note that $f$ is an isomorphism of $\CinfM$-modules and a homeomorphism by what was said so far. The composition $f$ is depicted in the following diagram.
\[
\xymatrix{
*\txt{$(\Gamma(E) \otimes_\CinfM^\pi \Gamma(F)) \oplus (\Gamma(E) \otimes_\CinfM^\pi \Gamma(F')) \oplus$\\$(\Gamma(E') \otimes_\CinfM^\pi \Gamma(F)) \oplus (\Gamma(E') \otimes_\CinfM^\pi \Gamma(F'))$}
\ar[d]^{\rho^{-1}} \\
(\Gamma(E) \oplus \Gamma(E')) \otimes_\CinfM^\pi (\Gamma(F) \oplus \Gamma(F')) \ar[d]^{\psi^{-1}} \\
\Gamma(E \oplus E') \otimes_\CinfM^\pi \Gamma(F \oplus F') \ar[d]^{\varphi^{-1}} \\
\Gamma((E \oplus E') \otimes (F \oplus F')) \ar[d]^{\lambda} \\
\Gamma(E \otimes F) \oplus \Gamma(E \otimes F') \oplus \Gamma(E' \otimes F) \oplus \Gamma(E' \otimes F')
}
\]

From this we obtain
\begin{align*}
(\lambda \circ & \varphi^{-1} \circ \psi^{-1} \circ \rho^{-1})
\begin{multlined}[t]
(s_1 \otimes_\CinfM^\pi t_1, s_2 \otimes_\CinfM^\pi t_1',\\
s_1' \otimes_\CinfM^\pi t_2, s_2' \otimes_\CinfM^\pi t_2')
\end{multlined} \\
& = (\lambda \circ \varphi^{-1} \circ \psi^{-1})
\begin{multlined}[t]
((s_1,0) \otimes_\CinfM^\pi (t_1,0) + (s_2,0) \otimes_\CinfM^\pi (0,t_1')\\
 + (0,s_1') \otimes_\CinfM^\pi (t_2,0) + (0,s_2') \otimes_\CinfM^\pi (0,t_2'))
     \end{multlined} \\
&= (\lambda \circ \varphi^{-1})
\begin{multlined}[t]( [x \mapsto (s_1(x), 0)] \otimes_\CinfM^\pi [ x \mapsto (t_1(x), 0)]\\
 + [x \mapsto (s_2(x), 0)] \otimes_\CinfM^\pi [ x \mapsto (0, t_1'(x))]\\
 + [x \mapsto (0, s_1'(x))] \otimes_\CinfM^\pi [x \mapsto (t_2(x), 0)]\\
 + [x \mapsto (0, s_2'(x))] \otimes_\CinfM^\pi [x \mapsto (0, t_2'(x))])
 \end{multlined} \\
& = \lambda
\begin{multlined}[t]
( [ x \mapsto (s_1(x), 0) \otimes (t_1(x), 0) ] + [x \mapsto (s_2(x), 0) \otimes (0, t_1'(x))] \\
+ [x \mapsto (0, s_1'(x)) \otimes (t_2(x), 0)] + [x \mapsto (0, s_2'(x)) \otimes (0, t_2'(x))]) 
\end{multlined}\\
&= \begin{multlined}[t]( [x \mapsto s_1(x) \otimes t_1(x) ], [x \mapsto s_2(x) \otimes t_1'(x)],\\
 [x \mapsto s_1'(x) \otimes t_2(x)], [x \mapsto s_2'(x) \otimes t_2'(x)] ).
\end{multlined}
\end{align*}

This means we can write $f = (f_1, f_2, f_3, f_4)$ with $f_1$ mapping $\Gamma(E) \otimes_\CinfM^\pi \Gamma(F)$ to $\Gamma(E \otimes F)$ and analogously for the other components. Because $f$ is bijective all $f_i$ have to be (\cite[Chapter II \S 1.6 Corollary 1 to Proposition 7]{Bourbaki}). As $f$ is a homeo\-morphism it follows immediately that all $f_i$ are homeomorphisms.
\end{proof}

Theorem \ref{sectens} implies that $\Gamma(E) \otimes_\CinfM^\pi \Gamma(F)$ is a Fr\'echet space.

Furthermore we obtain a homeomorphism for spaces of sections supported in a fixed compact set $K \csub M$. By Lemma \ref{topgleich} we thus have 
\begin{gather*}
\Gamma(E \otimes F) \cong \Gamma(E) \otimes_\CinfM^\pi \Gamma(F) = \Gamma(E) \otimes_\CinfM^\beta \Gamma(F) \\
\Gamma_{c,K}(E \otimes F) \cong \Gamma_{c,K}(E) \otimes_\CinfM^\pi \Gamma(F) = \Gamma_{c,K}(E) \otimes_\CinfM^\beta \Gamma(F).
\end{gather*}

We now prove the corresponding isomorphism for spaces of compactly supported sections.

\begin{remark}\label{algprop}
The validity of the following results is immediate from the fact that the functors $\mathunderscore \otimes_\CinfM^\beta \Gamma(F)$ and $\mathunderscore \times \Gamma(F)$ have as right adjoint the functor $\Linb_\CinfM(\Gamma(F), \mathunderscore)$, but we will explicitly prove them.
\end{remark}



\begin{lemma}\label{lfprod}Let a locally convex space $E$ be the strict inductive limit of a sequence of subspaces $E_n$ with embeddings $\iota_n\colon E_n \to E$ and let $F$ and $G$ be arbitrary locally convex spaces. Then a bilinear mapping $f\colon E \times F \to G$ is bounded if and only if all $f \circ (\iota_n \times \id)\colon E_n \times F \to G$ are bounded.
\end{lemma}
\begin{proof}Necessity is clear. For sufficiency, let $B \subseteq E \times F$ be bounded. As the canonical projections $\pi_1$ onto $E$ and $\pi_2$ onto $F$ are bounded $B_1 \coleq \pi_1(B)$ and $B_2 \coleq \pi_2(B)$ are bounded and $B$ is contained in the bounded set $B_1 \times B_2$. Because $B_1$ is bounded it is contained in some $E_n$, thus by assumption $f(B) \subseteq f(B_1 \times B_2) = f ( \iota_n (B_1) \times B_2) = f \circ (\iota_n \times \id)(B_1 \times B_2)$ is bounded.
\end{proof}


\begin{theorem}\label{seciso2}There is a bornological $C^\infty(M)$-module isomorphism
\[ \Gamma_c(E) \otimes_\CinfM^\beta \Gamma(F) \cong \Gamma_c(E \otimes F). \]
\end{theorem}
 \begin{proof}
Consider the following diagram.
\[
 \xymatrix{
\Gamma_{c,K}(E) \times \Gamma(F) \ar[d]_{\iota_K \times \id} \ar[r]_{\otimes_\CinfM^\beta} & \Gamma_{c,K}(E) \otimes_\CinfM^\beta \Gamma(F) \ar[dd]^\varphi \ar[dddl]_(0.4){f_K} \\
\Gamma_c(E) \times \Gamma(F) \ar[dd]_-{\otimes_\CinfM^\beta} \ar[ddr]^(0.54){\tilde h}|!{[ur];[dd]}\hole|!{[dr];[dd]}\hole &  \\
& \Gamma_{c,K}(E \otimes F) \ar[d]^{\iota_K'} \ar[dl]_(0.55){g_K}\\
\Gamma_c(E) \otimes_\CinfM^\beta \Gamma(F) \ar@<-0.5ex>[r]_-h & \ar@<-0.5ex>[l]_-g \Gamma_c(E \otimes F) \\
}
\]
Here $\iota_K: \Gamma_{c,K}(E) \to \Gamma_c(E)$ and $\iota_K': \Gamma_{c,K}(E \otimes F) \to \Gamma_c(E \otimes F)$ are the inclusion mappings.
For $K \csub M$ the $\CinfM$-bilinear bounded mapping $\otimes_\CinfM^\beta \circ (\iota_K \times \id)$ by Corollary \ref{bornisoalin} induces a bounded (and thus continuous) linear mapping $f_K$. Because $\varphi$ is a homeomorphism there is a corresponding linear continuous mapping $g_K \coleq \varphi^{-1} \circ f_K$. Because $\Gamma_c(E \otimes F)$ is the strict inductive limit of the spaces $\Gamma_{c,K}(E \otimes F)$ and for different $K$ the mappings $g_K$ are compatible with each other there is a unique continuous linear mapping $g$ such that $g \circ \iota_K' = g_K$.

By Lemma \ref{lfprod} the bilinear mapping $\tilde h$ defined by $\tilde h(s,t)(x) \coleq s(x) \otimes t(x)$ is bounded because all $\tilde h \circ (\iota_K \times \id) = \iota_K' \circ \varphi \circ \otimes_\CinfM^\beta$ are bounded, thus a unique bounded linear mapping $h$ completing the diagram exists. It is easily verified that $g$ and $h$ are inverse to each other, which completes the proof.
\end{proof}

Because the spaces involved are bornological one can also say that the isomorphism of the previous theorem is a homeomorphism.

\begin{remark}Similarly one can obtain
\[ \Gamma(E) \otimes^\beta_{\CinfM} \Gamma_c(F) \cong \Gamma_c(E) \otimes^\beta_{\CinfM}\Gamma_c(F) \cong \Gamma_c(E \otimes F). \]
\end{remark}

Note that Lemma \ref{lfprod} and thus Theorem \ref{seciso2} only work in the bornological setting but not in the topological one.

\section{Distributions on manifolds}\label{sec_distrib}

In this chapter we will finally define the space of tensor distributions and give bornologically isomorphic representations. For additional information on distributions on manifolds we refer to \cite[Section 3.1]{GKOS}. In what follows $\Vol(M)$ denotes the volume bundle over $M$ (\cite[Definition 3.1.1]{GKOS}).

\begin{definition} The space of distributions a manifold $M$ is defined as 
\[ \DpM \coleq [\Gamma_c(M, \Vol(M))]' \]
and the space of tensor distributions of rank $(r,s)$ on $M$ as
\[ \DprsM \coleq [\Gamma_c(M, \TsrM \otimes \Vol(M))]'. \]
The spaces of compactly supported sections are equipped with the (LF)-topology described in Section \ref{sec_sectop} which is bornological, thus these are exactly the bounded linear functionals. $\DpM$ and $\DprsM$ carry the strong dual topology (\cite[Chapter 19]{Treves}).
\end{definition}

The following is the bornological version of \cite[Theorem 3.1.12]{GKOS}.

\begin{theorem}\label{distiso}There are bornological $\CinfM$-module isomorphisms
\begin{align}
\label{dist_eins}\DprsM &\cong (\cTsrM \otimes_\CinfM^\beta \Gamma_c(M, \Vol(M)))' \\
\label{dist_zwei}&\cong \Linb_{\CinfM}(\cTsrM, \DpM) \\
\label{dist_drei}&\cong \cTrsM \otimes_\CinfM^\beta \DpM.
\end{align}
\end{theorem}
\begin{proof}\eqref{dist_eins} is clear from the bornological isomorphism of $\CinfM$-modules
\[ \Gamma_c(M, \TsrM \otimes \Vol(M)) \cong \cTsrM \otimes_\CinfM^\beta \Gamma_c(M, \Vol(M)) \]
given by Theorem \ref{seciso2}. As both spaces are bornological it is also an isomorphism of topological vector spaces, thus the duals are homeomorphic (\cite[Chapter 23]{Treves}).

\eqref{dist_eins} $\leftrightsquigarrow$ \eqref{dist_zwei} is clear from Proposition \ref{tpiso}.

For \eqref{dist_zwei} $\leftrightsquigarrow$ \eqref{dist_drei} consider the map
\[ \theta_\cTsrM\colon \cTsrM^* \otimes_{\CinfM} \DpM \to \Lin_{\CinfM}(\cTsrM, \DpM) \]
induced by the bilinear map
\begin{equation}\label{baldfertig}
 \begin{aligned}
  \cTsrM^* \times \DpM &\to \Lin_\CinfM(\cTsrM, \DpM) \\
  (u^*, v) &\mapsto [u \mapsto u^*(u) \cdot v].
 \end{aligned}
\end{equation}
Because $\cTsrM$ is finitely generated and projective it is a direct summand of a free finitely generated $\CinfM$-module $F$ with injection $\iota$ and projection $\pi$. By \cite[2.23]{GHV} there exists a vector bundle $C \to M$ such that $\TsrM \oplus C$ is trivial, thus we can take $F = \cTsrM \oplus \Gamma(C)$. Note that duals of $F$ and $\cTsrM$ here are always meant with respect to the $\CinfM$-module structure.
By standard methods (cf.\ the proof of \cite[Theorem 14.10]{blyth}) one obtains the commutative diagram
\[
 \xymatrix{
F^* \otimes_{\CinfM} \DpM \ar[r]^-{\iota^* \otimes \id} \ar[d]_{\theta_F}& \cTsrM^* \otimes_{\CinfM} \DpM \ar[r]^-{\pi^* \otimes \id} \ar[d]_{\theta_\cTsrM} & F^* \otimes_{\CinfM} \DpM \ar[d]_{\theta_F} \\
\Lin_{\CinfM}(F,\DpM) \ar[r]_-{\iota^\mathrm{t}} & \Lin_{\CinfM}(\cTsrM,\DpM) \ar[r]_-{\pi^\mathrm{t}} & \Lin_{\CinfM}(F,\DpM)
}
\]
with mappings
\begin{gather*}
 \iota^*\colon F^* \to \cTsrM^*,\ u^* \mapsto u^* \circ \iota \\
 \pi^*\colon \cTsrM^* \to F^*,\ u^* \mapsto u^* \circ \pi \\
 \iota^\mathrm{t}\colon \Lin_{\CinfM}(F, \DpM) \to \Lin_{\CinfM}(\cTsrM, \DpM),\ \ell \mapsto \ell \circ \iota \\
 \pi^\mathrm{t}\colon \Lin_{\CinfM}(\cTsrM, \DpM) \to \Lin_{\CinfM}(F, \DpM),\ \ell \mapsto \ell \circ \pi
\end{gather*}
where $\iota^* \otimes \id$ and $\iota^\mathrm{t}$ are surjective while $\pi^* \otimes \id$ and $\pi^\mathrm{t}$ are injective.

The inverse of $\theta_F$ can be given explicitly because $F$ is free and finitely generated. Let $\{b_1,\dotsc, b_n\}$ be a basis of $F$ and $\{b_1^*, \dotsc, b_n^*\}$ the corresponding dual basis of $F^*$. For $\ell \in \Lin_\CinfM(F, \DpM)$ we have
\[ \theta_F^{-1}(\ell) = \sum_{\mathclap{i=1\dotsc n}} b_i^* \otimes \ell(b_i) \in F^* \otimes_\CinfM \DpM. \]
This implies that also $\theta_\cTsrM$ is an isomorphism, its inverse is given by the composition  $(\iota^* \otimes \id) \circ \theta_F^{-1} \circ \pi^\mathrm{t}$.

As \eqref{baldfertig} is bounded from $\cTsrM' \times \DpM$ into $\Linb_\CinfM(\cTsrM, \DpM)$ the induced mapping $\theta_{\cTsrM}\colon \cTsrM' \otimes_\CinfM^\beta \DpM \to \Linb_\CinfM(\cTsrM, \DpM)$ is bounded and linear. Because $\iota$ and $\pi$ obviously are continuous all mappings in the following diagram are bounded. 
\[
 \xymatrix{
F' \otimes^\beta_\CinfM \DpM \ar[r]^-{\iota^* \otimes \id} \ar[d]_{\theta_F}& \cTsrM' \otimes^\beta_\CinfM \DpM \ar[r]^-{\pi^* \otimes \id} \ar[d]_{\theta_\cTsrM} & F' \otimes^\beta_\CinfM \DpM \ar[d]_{\theta_F} \\
\Linb_{\CinfM}(F,\DpM) \ar[r]_-{\iota^\mathrm{t}} & \Linb_{\CinfM}(\cTsrM,\DpM) \ar[r]_-{\pi^\mathrm{t}} & \Linb_{\CinfM}(F,\DpM)
}
\]
Concluding, $\theta_F^{-1}\colon \ell \mapsto \sum_i b_i^* \otimes_\CinfM^\beta \ell(b_i)$ is bounded into $F' \otimes_\CinfM^\beta \DpM$ whence $\theta_\cTsrM^{-1} = (\iota^* \otimes \id) \circ \theta_F^{-1} \circ \pi^\mathrm{t}$ also is bounded.
\end{proof}

 \begin{lemma}\label{dpmult} Multiplication $\CinfM \times \DpM \to \DpM$, $(f,T) \mapsto f \cdot T = [\omega \mapsto \langle T, f \cdot \omega \rangle ]$ is bounded.
 \end{lemma}
 \begin{proof}
As the bornology of $\DpM$ consists of all weakly bounded sets we only have to verify that for $B_1 \subseteq \CinfM$ and $B_2 \subseteq \DpM$ both bounded $\{\, \langle T, f \cdot \omega \rangle\ |\ f \in B_1, T \in B_2\,\}$ is bounded for each $\omega \in \ocM$, which follows because $\{\, f \cdot \omega\ |\ f \in B_1\,\}$ is bounded in $\ocM$ and $B_2$ is uniformly bounded on bounded sets.
\end{proof}

\begin{remark}
\begin{enumerate}
 \item[(i)] A result analogous to Theorem \ref{distiso} is obviously valid for distributions of arbitrary density character taking values in any vector bundle instead of the tensor bundle (cf.~\cite[Definition 3.1.4]{GKOS}).
 \item[(ii)] Because multiplication of distributions is not jointly continuous (\cite{kucera}) the proof of Theorem \ref{distiso} fails for the projective tensor product. 
\end{enumerate}
\end{remark}

\section*{Acknowledgments}

This research has been supported by START-project Y237 and project P20525 of
the Austrian Science Fund and the Doctoral College 'Differential Geometry and Lie
Groups' of the University of Vienna.

\bibliographystyle{halpha}
\bibliography{master}

\end{document}